\documentclass[10pt]{article}
\usepackage{mathrsfs}
\usepackage{amsfonts}
\usepackage{titlesec}
\usepackage{enumerate}
\usepackage[toc,page,title,titletoc,header]{appendix}
\usepackage{amsthm,amsmath,amssymb,anysize,amscd}
\usepackage[numbers,sort&compress]{natbib}
\usepackage[all]{xy}
\usepackage{amsthm}
\theoremstyle{definition}
\newtheorem{lemma}{Lemma}
\newtheorem{theorem}[lemma]{Theorem}

\newtheorem{proposition}[lemma]{Proposition}
\newtheorem{definition}[lemma]{Definition}

\newcommand{\periodafter}[1]{#1.}
\titleformat{\subsection}[runin]
{\normalfont\bfseries}{\thesubsection}{0.5em}{\periodafter}
\marginsize{3.5cm}{3.5cm}{3.6cm}{3cm}%×óÓÒÉÏÏÂ
\numberwithin{equation}{section}

\renewcommand{\thesubsection}{\arabic{section}.\arabic{subsection}.}

\newcounter{RomanNumber}

\title{\textsf{The multiplier and cohomology of \\ Lie   superalgebras}}
\author{\textsc{Yang Liu$^{1}$ and
  \textsc{Wende Liu}$^{1,2}$}\footnote{Correspondence:  wendeliu@ustc.edu.cn (W. D. Liu). The first named author (Y. Liu) was supported by the NSF
  of China (11501151) and the NSF of Heilongjiang Province (A2015003, A2017005). The second named author (W. D. Liu) was supported by the NSF
  of China (11171055, 11471090).}\;\;\;
  \\
  \\
  \ \ \textit{$^{1}$Department of Mathematics},
  \textit{Harbin Institute of Technology} \\
  \textit{Harbin 150006, China}
  \\
   \ \ \textit{$^{2}$School of Mathematical Sciences},
  \textit{Harbin Normal University} \\
  \textit{Harbin 150025, China}
  }
\date{ }
\begin{document}
\maketitle
\begin{quotation}
\small\noindent \textbf{Abstract}:
In this paper, all (super)algebras are over a field $\mathbb{F}$ of characteristic different from $2, 3$. We construct the so-called  5-sequences of cohomology for central extensions of a Lie superalgebra and prove that they are exact. Then we prove that the multipliers of a Lie superalgebra are  isomorphic to the second cohomology group with coefficients in the trivial module for the Lie superalgebra under consideration.

\vspace{0.2cm} \noindent{\textbf{Keywords}}:  multiplier, cohomology, 5-sequence

\vspace{0.2cm} \noindent{\textbf{Mathematics Subject Classification 2010}}: 17B05, 17B30, 17B56
\end{quotation}

\setcounter{section}{0}
\section{Introduction}
In the early 1900's, multipliers and covers were first studied by I. Schur in finite group theory, where the multiplier is defined to be the second cohomology group with coefficient in $\mathbb{C}^{*}$ \cite{1}. Let $0\longrightarrow R\longrightarrow F\longrightarrow G\longrightarrow 0$ be a free presentation of a group $G$. Then it can be shown that the multiplier of $G$ is isomorphic to $(F^{2}\cap R)/[R, F]$ by Hopf's formula \cite{2}. The notions of multipliers were naturally generalized to Lie algebra case. Let $0\longrightarrow R\longrightarrow F\longrightarrow L\longrightarrow 0$ be a free presentation of a Lie algebra $L$ over a field $\mathbb{F}$. The multiplier of $L$ is defined to be the factor Lie algebra $(F^{2}\cap R)/[R, F]$, denoted by $\mathcal{M}(L)$.
P. Batten proved that the multiplier for a finite dimensional Lie algebra is isomorphic to $\mathrm{H}^{2}(L, \mathbb{F})$, the second cohomology group of $L$, where $L$ acts trivially on $\mathbb{F}$ \cite{3, 4}. Since the multiplier of a Lie algebra $L$ is abelian, one may pay attention to the dimension of the multiplier. K. Moneyhun gave an equivalent definition of the multiplier by the maximal defining pairs and proved that for a Lie algebra $L$ of dimension $n$, it must be that $\mathrm{dim}\mathcal{M}(L)=n(n-1)/2-t(L)$ for some non-negative integer $t(L)$ \cite{5}. A. R. Salemkar, V. Alamian and H. Mohammadzadeh studied some properties of multipliers and covers of Lie algebras \cite{6}.  A. R. Salemkar, B. Edalatzadeh and H. Mohammadzadeha gave a complete structure characterizer of covers of perfect Lie algebras \cite{7}. There were many results about multipliers of nilpotent Lie algebras \cite{8, 9, 10, 11, 12, 13, 14, 15, 16}. A classification of nilpotent Lie algebras for which $t(L)$ takes certain
values was obtained \cite{8, 9, 10}. One might find various bounds of the dimension of the multipliers of nilpotent Lie algebras \cite{11, 12, 13, 14, 15, 16}. S. Nayak generalized the definition of multipliers and covers of Lie algebras to Lie superalgebras case and
introduced the concepts of isoclinism in Lie superalgebras \cite{17, 18}.  Y. L. Zhang and W. D. Liu introduced the concept of (super-)multiplier-rank for Lie superalgeras and classified  the nilpotent Lie superalgebras of multiplier-rank $\leq$ 2 \cite{19}. X. X. Miao and W. D. Liu introduced the definition of stem extensions for Lie superalgebras and proved that
multipliers and  covers  always exist for a Lie superalgebra  and they are unique up to Lie superalgebra isomorphism \cite{20}.

 In this paper, we first establish the 5-sequences of cohomology for central extensions of Lie superalgebras. Then we prove they are exact and we prove that the multipliers of a Lie superalgebra are  isomorphic to the second cohomology group with coefficients in the trivial module for the Lie superalgebra under consideration.

\section{Notation and preparatory results}
In this paper, all (super)algebras are over a field $\mathbb{F}$ of characteristic different from $2, 3$. For a Lie superalgebra $L$, we use $[~, ~]$ to denote the operation of $L$. For a homogeneous element $x\in L$, we use $|x|$ to denote the degree of $x$.
The symbol $|x|$  implies that $x$ has been assumed to be  a  homogeneous element.

Let us recall the notion of extensions of Lie superalgebras and some basic properties. Note that a Lie superalgebra homomorphism is always an even  homomorphism and an ideal of a Lie superalgebra is always a $\mathbb{Z}_{2}$-graded ideal. An extension of a Lie superalgebra $L$ by  $M$ is a short exact sequence of Lie superalgebra homomorphisms:
$$0\longrightarrow M\stackrel{\alpha}{\longrightarrow} K \stackrel{\beta}{\longrightarrow} L \longrightarrow 0.$$
We usually identify $M$ with a subalgebra of $K$ and omit the embedding map $\alpha$.

Suppose we are given two
extensions  $$e:0\longrightarrow M\longrightarrow K\stackrel{\pi}{\longrightarrow}L\longrightarrow 0,$$ $$e':0\longrightarrow M\longrightarrow K'\stackrel{\pi'}{\longrightarrow}L\longrightarrow 0.$$ We called $e$ equivalent to $e'$ if there exists a Lie superalgebra homomorphism $\rho: K\longrightarrow K'$ such that the following diagram is commutative.
$$\xymatrix{0\ar[r]&M\ar@{=}[d]_{\mathrm{id}}\ar[r]&K\ar[d]_{\rho}\ar[r]^{\pi}&L\ar@{=}[d]_{\mathrm{id}}\ar[r]&0\\
0\ar[r]&M\ar[r]&K'\ar[r]^{\pi'}&L\ar[r]&0.}$$

A short exact sequence
$0\longrightarrow M\longrightarrow K\longrightarrow L\longrightarrow 0$ of Lie superalgebras is a central extension of $L$ if $M\subseteq \mathrm{Z}(K)$, where $\mathrm{Z}(K)$ denotes the center of $M$. Note that the Lie superalgebra homomorphism is a Lie superalgebra isomorphism for equivalent central extensions. Clearly, the equivalence of central extensions is an equivalent relationship and  the set of all central extensions of $L$ by $M$ is divided into equivalence classes of central extensions.

Let $e:0\longrightarrow M\longrightarrow K\stackrel{\pi}{\longrightarrow}L\longrightarrow 0$ be a central extension.
Take $\mu: L\longrightarrow K$ to be an even linear map such that $\pi\circ\mu=\mathrm{id}_{L}$, that is, $\mu$ is a section of $\pi$. Define $f(x, y)=[\mu(x), \mu(y)]-\mu([x,y])$, where $x, y\in L$. We will show that $f$ is a bilinear map from $L\times L$ to $M$, which is called a factor set defined by $\mu$. Let us give some properties of a factor set.

\begin{proposition}\label{pro211}
Let $e:0\longrightarrow M\longrightarrow K\stackrel{\pi}{\longrightarrow}L\longrightarrow 0$ be a central extension , $\mu$  a section of $\pi$ and $f$ a factor set defined by $\mu$. Then we have:
\begin{itemize}
\item[(i)]
 $f$ is an even bilinear map from $L\times L$ to $M$,
\item[(ii)]
 $f(x, y)=-(-1)^{|x||y|}f(y, x)$ for all $x, y\in L$,
\item[(iii)]
$(-1)^{|x||z|}f([x, y], z)+(-1)^{|y||x|}f([y, z], x)+(-1)^{|z||y|}f([z, x], y)=0$ for all $x, y, z \in L$.
 \end{itemize}
\end{proposition}
\begin{proof}
$(\mathrm{i})$ Clearly, $(\pi\circ f)(x, y)=\pi([\mu(x), \mu(y)])-\pi(\mu([x,y]))$, where $x, y \in L $. Since $\pi\circ\mu=\mathrm{id}_{L}$, we have $(\pi\circ f)(x, y)=0$ and $f(x,y)\in \ker\pi=M$. Note that $\mu$ is an even linear map. $(\mathrm{i})$ holds.

$(\mathrm{ii})$  Note that $\mu$ is even. We have \begin{align*}
                                   f(x, y)=  &[\mu(x), \mu(y)]-\mu([x,y])  \\
                                    = & -(-1)^{|\mu(x)||\mu(y)|}[\mu(y),\mu(x)]-(-(-1)^{|x||y|}\mu([y,x]))\\
                                    = & -(-1)^{|x||y|}f(y, x).
                                  \end{align*}

$(\mathrm{iii})$
According to the super Jacobi identity in $L$, we have
\begin{equation*}
    (-1)^{|x||z|}\mu([[x, y], z])+(-1)^{|y||x|}\mu([[y, z], x])+(-1)^{|z||y|}\mu([[z, x], y])=0.
\end{equation*}
Since $[\mu([x, y]), \mu(z)]=[[\mu(x), \mu(y)]-f(x, y), \mu(z)]$ and $f(x, y)\in M\subseteq \mathrm{Z}(K)$, we have
\begin{align*}
   & (-1)^{|x||z|}[\mu([x, y]), \mu(z)]+(-1)^{|y||x|}[\mu([y, z]), \mu(x)]+(-1)^{|z||y|}[\mu([z, x]), \mu(y)] \\
  =& (-1)^{|x||z|}[[\mu(x), \mu(y)], \mu(z)]+(-1)^{|y||x|}[[\mu(y), \mu(z)], \mu(x)]+(-1)^{|z||y|}[[\mu(z, \mu(x)], \mu(y)]\\
  =& 0.
\end{align*}
This implies that $(\mathrm{iii})$ holds.
\end{proof}

Let $e: 0\longrightarrow M\longrightarrow K\stackrel{\pi}{\longrightarrow}L\longrightarrow 0$ be a central extension. Suppose $\mu$ and $\mu'$  are sections of $\pi$. Let $f$ and $g$ be factor sets defined by $\mu$ and $\mu'$, respectively. Then we have factor sets for the central extension $e$ differ by a linear map $(\mu-\mu'): L\longrightarrow M$. Suppose we are given two central extensions $e_{1}, e_{2}$ such that $e_{1}$ is equivalent to $e_{2}$, let $\rho: K_{1}\longrightarrow K_{2}$ be a Lie superalgebra homomorphism, $\mu_{1}$ and $\mu_{2}$ two sections of $\pi_{1}$ and $\pi_{2}$,
$f_{1}$ and $f_{2}$ factor sets defined by $\mu_{1}$ and $\mu_{2}$. Clearly, $\rho\circ\mu_{1}$ is also a section of $\pi_{2}$. Then $\rho\circ f_{1}(x, y)-f_{2}(x, y)=-(\rho\circ\mu_{1}-\mu_{2})([x, y])$. Since $\rho|_{M}=\mathrm{id}_{M}$, we have $\rho\circ f_{1}=f_{1}$. A
direct calculation shows that factor sets for equivalent central extensions $e_{1}$ and $e_{2}$ differ by a linear map $(\rho\circ\mu_{1}-\mu_{2}): L\longrightarrow M$.

Analogous to the group case \cite[Proposition 10.3.5]{21}, we have the following lemma.
\begin{lemma}\label{lem212}Let $0\longrightarrow R\stackrel{i}{\longrightarrow} F\stackrel{\pi}{\longrightarrow}L\longrightarrow 0$ be a free presentation of a Lie superalgebra $L$
and
$0\longrightarrow M\stackrel{i}{\longrightarrow} K\stackrel{\nu}{\longrightarrow}\overline{L} \longrightarrow 0$ an extension of a Lie superalgebra $\overline{L}$. Let $\alpha: L\longrightarrow \overline{L}$ be a Lie superalgebra homomorphism. Then there exists a Lie superalgebra homomorphism
$\bar{\beta}: F^{2}/[R, F]\longrightarrow K^{2}/[M, K]$ induced by $\beta: F\longrightarrow K$ such that  the following diagram is commutative:
$$\xymatrix{0\ar[r]&\dfrac{F^{2}\cap R}{[R, F]}\ar[d]_{\bar{\beta}_{1}}\ar[r]^{\bar{i}}&\dfrac{F^{2}}{[R,  F]}\ar[d]_{\bar{\beta}}\ar[r]^{\bar{\pi}}&L^{2}\ar[d]_{\bar{\alpha}}\ar[r]&0\\
0\ar[r]&\dfrac{K^{2}\cap M}{[M, K]}\ar[r]^{\bar{i}}&\dfrac{K^{2}}{[M, K]}\ar[r]^{\bar{\nu}}&\overline{L}^{2}\ar[r]&0,}$$
where the maps $\bar{i}$ and $\bar{\pi}$ in the rows are the induces maps, while $\bar{\beta}_{1}$ and $\bar{\alpha}$ are the restrictions of $\beta$
and $\alpha$, respectively. Moreover, if $\bar{\gamma}: F^{2}/[R, F]\longrightarrow K^{2}/[M, K]$ is also a Lie superalgebra homomorphism induced by $\gamma: F\longrightarrow K$, then the homomorphism $\bar{\gamma}$ is the same as $\bar{\beta}$.
\end{lemma}

\begin{proof}
By \cite[Theorem 2.5]{20}, we can get a Lie superalgebra homomorphism $\bar{\beta}: F^{2}/[R, F]\longrightarrow K^{2}/[M, K]$ induced by $\beta: F\longrightarrow K$ such that the following diagram is commutative:
$$\xymatrix{0\ar[r]&\dfrac{R}{[R, F]}\ar[d]_{\bar{\beta}_{1}}\ar[r]^{\bar{i}}&\dfrac{F}{[R,  F]}\ar[d]_{\bar{\beta}}\ar[r]^{\bar{\pi}}&L\ar[d]_{\alpha}\ar[r]&0\\
0\ar[r]&\dfrac{K}{[M, K]}\ar[r]^{\bar{i}}&\dfrac{K}{[M, K]}\ar[r]^{\bar{\nu}}&\overline{L}\ar[r]&0,}$$
where $\bar{\beta}_{1}$ is the restrictions of $\bar{\beta}$. Since $\pi(F^{2})\subseteq L^{2}$ and $\nu(K^{2})\subseteq \overline{L}^{2}$
, we have $\bar{\beta}(F^{2}/[R, F])\subseteq K^{2}/[M, K]$ and the diagram in the statement of the lemma is commutative.

If $\bar{\gamma}: F^{2}/[R, F]\longrightarrow K^{2}/[M, K]$ is also a Lie superalgebra homomorphism induced by $\gamma: F\longrightarrow K$ such that the following diagram is commutative:
$$\xymatrix{0\ar[r]&\dfrac{R}{[R, F]}\ar[d]_{\bar{\gamma}_{1}}\ar[r]^{\bar{i}}&\dfrac{F}{[R,  F]}\ar[d]_{\bar{\gamma}}\ar[r]^{\bar{\pi}}&L\ar[d]_{\alpha}\ar[r]&0\\
0\ar[r]&\dfrac{K}{[M, K]}\ar[r]^{\bar{i}}&\dfrac{K}{[M, K]}\ar[r]^{\bar{\nu}}&\overline{L}\ar[r]&0,}$$
where $\bar{\gamma}_{1}$ is the restrictions of $\bar{\gamma}$. Let $\bar{x}\in F/[R,  F]$. Then $\bar{\nu}(\bar{\gamma}(\bar{x}))=\alpha(\bar{\pi}(\bar{x}))=\bar{\nu}(\bar{\beta}(\bar{x}))$. Hence, $\bar{\nu}(\bar{\gamma}(\bar{x})-\bar{\beta}(\bar{x}))=0$ and $\bar{\gamma}(\bar{x})=\bar{\beta}(\bar{x})+\bar{m}_{x}$ for some $\bar{m}_{x}\in M/[M, K]$. Since $M/[M, K]$ is contained in the center of $K/[M, K]$, we have $\bar{\gamma}([\bar{x}, \bar{y}])=[\bar{\gamma}(\bar{x}), \bar{\gamma}(\bar{y})]=[\bar{\beta}(\bar{x})+\bar{m}_{x}, \bar{\beta}(\bar{y})+\bar{m}_{y}]=[\bar{\beta}(\bar{x}), \bar{\beta}(\bar{y})]=\bar{\beta}([\bar{x}, \bar{y}])$. This implies that $\bar{\beta}=\bar{\gamma}$ when restricted to $F^{2}/[R, F]$, and so also
to $(F^{2}\cap R)/[R, F]$.
\end{proof}

\begin{proposition}\label{pro213}
Let $0\longrightarrow R\stackrel{i}{\longrightarrow} F\stackrel{\pi}{\longrightarrow}L\longrightarrow 0$ and $0\longrightarrow R'\stackrel{i'}{\longrightarrow} F'\stackrel{\pi'}{\longrightarrow}L\longrightarrow 0$ be two free presentations of a Lie superalgebra $L$. Then the Lie superalgebras $(F^{2}\cap R)/[R, F]$ and $(F'^{2}\cap R')/[R', F']$ are naturally isomorphic.
\end{proposition}
\begin{proof}
According to Lemma \ref{lem212}, for the identity map $\mathrm{id}_{L}$, we can get a unique Lie superalgebra homomorphism $\bar{\rho}: (F^{2}\cap R)/[R, F]\longrightarrow (F'^{2}\cap R')/[R', F']$ which is induced by a Lie superalgebra homomorphism $\rho: F\longrightarrow F'$ and also a unique Lie superalgebra homomorphism $\bar{\rho}':(F'^{2}\cap R')/[R', F']\longrightarrow (F^{2}\cap R)/[R, F]$ which is induced by a Lie superalgebra homomorphism $\rho': F'\longrightarrow F$. Thus, $\rho'\circ\rho$ and $\mathrm{id}_{F}$ are both Lie superalgebra homomorphism
from $F$ to itself and so they induce the same Lie superalgebra homomorphism from $(F^{2}\cap R)/[R, F]$ to itself. This implies that $\bar{\rho}'\circ \bar{\rho}= \mathrm{id}_{F}$. Similarly, we have $\bar{\rho}\circ\bar{\rho}'=\mathrm{id}_{F'}$.
This shows $\bar{\rho}$ and $\bar{\rho}$ are isomorphisms.
\end{proof}
 By Proposition \ref{pro213} the Lie superalgebra $(F^{2}\cap R)/[R, F]$ is independent of the choice of the free presentation, we have right to have the following definition.

\begin{definition}\label{def214}
Let $L$ be a Lie superalgebra and
$$0\longrightarrow R\longrightarrow F\stackrel{\pi}{\longrightarrow}L\longrightarrow 0$$ a free presentation of $L$. Then we call the factor Lie superalgebra $(F^{2}\cap R)/[R, F]$ the Schur multiplier of $L$, denoted by $\mathcal{M}(L)$.
\end{definition}

Next we introduce some basic concepts of the Lie superalgebra cohomplogy \cite{22, 23}.

Let $\mathbb{F}$ be a trivial $L$-module and $C^{0}(L, \mathbb{F})=\mathbb{F}$. For $n\geq 1$, let $C^{n}(L, \mathbb{F})$ be the $\mathbb{Z}_{2}$-graded vector space of all $n$-linear super skew symmetry functions $f:L\times \ldots\times L\longrightarrow \mathbb{F}$ such that $$f(x_{1}, \ldots, x_{i}, x_{i+1}, \ldots, x_{n})=-(-1)^{|x_{i}||x_{i+1}|}f(x_{1}, \ldots, x_{i+1}, x_{i}, \ldots, x_{n})$$
for $x_{1}, \ldots, x_{n}\in L$.

We define an $L$-module structure on $C^{n}(L, \mathbb{F})$ as follows:
\begin{align*}
                                     &(x\cdot f)(x_{1}, \ldots, x_{n})  \\
                                   = &-\sum\limits_{i=1}^{n}(-1)^{|x|(|f|+|x_{1}|+\ldots+|x_{i-1}|)}f(x_{1}, \ldots, x_{i-1}, [x, x_{i}], x_{i+1}, \ldots, x_{n})
 \end{align*}
for $x, x_{1}, \ldots, x_{n}\in L$ and $f\in C^{n}(L, \mathbb{F})$.

We define a differential operator $\delta: C^{n}(L, \mathbb{F})\longrightarrow C^{n+1}(L, \mathbb{F})$ as follows:
\begin{align*}
   & (\delta f)(x_{1}, \ldots, x_{n+1}) \\
  = & \sum\limits_{i<j}(-1)^{j+|x_{j}|(|x_{i+1}|+\ldots+|x_{j-1}|)}f(x_{1}, \ldots, x_{i-1}, [x_{i}, x_{j}], x_{i+1}, \ldots, \hat{x}_{j}, \ldots, x_{n+1})
\end{align*}
for $x_{1}, \ldots, x_{n+1}\in L$ and $f\in C^{n}(L, \mathbb{F})$, where the sign $\hat{}$ means that the element under it is omitted.

Let $$\mathrm{Z}^{n}(L, \mathbb{F})=\ker \delta=\{f\in C^{n}(L, \mathbb{F})\mid \delta f=0\},$$
$$\mathrm{B}^{n}(L, \mathbb{F})=\mathrm{Im}\delta=\{\delta(f)\mid f \in C^{n-1}(L, \mathbb{F})\}.$$ One can check that $\delta^{2}=0$. This implies that $\mathrm{B}^{n}(L, \mathbb{F})\subseteq \mathrm{Z}^{n}(L, \mathbb{F})$. Define $\mathrm{H}^{n}(L, \mathbb{F})=
\mathrm{Z}^{n}(L, \mathbb{F})/\mathrm{B}^{n}(L, \mathbb{F})$. The space $\mathrm{H}^{n}(L, \mathbb{F})$ is called the Lie superalgebra cohomology group with coefficient in $\mathbb{F}$. For $f\in \mathrm{Z}^{n}(L, \mathbb{F})$, we denote by $\bar{f}$ its image in $\mathrm{H}^{n}(L, \mathbb{F})$ under the canonical map.
We are interested in the case $n=2$.We have
\begin{align*}
 & -(-1)^{|x||z|}\delta f(x, y, z) \\
= & (-1)^{|x||z|}f([x, y], z)+(-1)^{|y||x|}f([y, z], x)+(-1)^{|z||y|}f([z, x], y)
\end{align*}
for $x, y, z\in L$ and $f\in C^{2}(L, \mathbb{F})$. Then for $f\in C^{2}(L, \mathbb{F})$, we have $f\in \mathrm{Z}^{2}(L, \mathbb{F})$ if and only if $(-1)^{|x||z|}f([x, y], z)+(-1)^{|y||x|}f([y, z], x)+(-1)^{|z||y|}f([z, x], y)=0$ and $f\in \mathrm{B}^{2}(L, \mathbb{F})$ if and only if there exists a linear map $\sigma: L \longrightarrow \mathbb{F}$ such that $f(x, y)=-\sigma([x, y])$. By Proposition \ref{pro211}, a factor set is a 2-cocycle. Futhermore,
factor sets for equivalent central extension rise to an element of the second cohomlogy group and this element is independent of the choice of the section. According to \cite[Theorem 26.2]{21}, the elements of $\mathrm{H}^{2}(L, \mathbb{F})$ are in 1-1 correspondence with equivalence classes of central extensions of $L$ by $M$.

\section{Main results}
Let $L$ be a Lie superalgebra and $H$ a central subalgebra of $L$. Then
$$0\longrightarrow H\stackrel{i}{\longrightarrow} L\stackrel{\pi}{\longrightarrow}L/H\longrightarrow 0$$
 be a central extension. Let $\mu$ be a section of $\pi$ and $\mathbb{F}$ a trivial $L$-module. We view $H$, $L$, $L/H$ as $L$-modules by adjoint representation and
we write $\mathrm{Hom}(M, N)$ for the set of $L$-module homomorphisms from $M$ to $N$.

Define
 \begin{align*}
                 \mathrm{Inf}:\mathrm{Hom}(L/H, \mathbb{F}) & \longrightarrow \mathrm{Hom}(L, \mathbb{F}) \\
                f & \longmapsto f\circ\pi.
\end{align*}
Since $\pi$ is a Lie superalgebra homomorphism, we have $f\circ\pi\in \mathrm{Hom}(L, \mathbb{F})$. Clearly, the map $\mathrm{Inf}$ is an even linear map.

Define
\begin{align*}
                 \mathrm{Res}:\mathrm{Hom}(L, \mathbb{F}) & \longrightarrow \mathrm{Hom}(H, \mathbb{F}) \\
                g & \longmapsto g\circ i.
               \end{align*}
Note that $i$ is a Lie superalgebra homomorphism. Clearly, the map $\mathrm{Res}$ is also an even linear map.

Define
\begin{align*}
                 \mathrm{Tra}:\mathrm{Hom}(H, \mathbb{F}) & \longrightarrow \mathrm{H}^{2}(L/H, \mathbb{F}) \\
               \alpha & \longmapsto \overline{\alpha\circ f}
               \end{align*}
where $f$ is a factor set defined by $\mu$. Let $\alpha\in \mathrm{Hom}(H, \mathbb{F})$. Since $f\in \mathrm{Z}^{2}(L/H, H)$, we have $\alpha\circ f\in \mathrm{Z}^{2}(L/H, \mathbb{F})$. Let $\mu'$ be also a section of $\pi$ and $g$ a factor set defined by $\mu'$. Recall that factor sets for the same central extension differ by a linear map, we have $\alpha\circ f-\alpha\circ g =-\alpha(\mu-\mu')([\bar{x}, \bar{y}])$. It follows $\alpha\circ f-\alpha\circ g\in \mathrm{B}^{2}(L/H, \mathbb{F})$ and $\overline{\alpha\circ f}=\overline{\alpha\circ g}$ in $H^{2}(L/H, \mathbb{F})$. Define $\mathrm{Tra}$ by $\mathrm{Tra}(\alpha)=\overline{\alpha\circ f}$.
Since $\mathrm{Tra}(k\alpha+\alpha')=\overline{(k\alpha+\alpha')\circ f}
=\overline{k\alpha\circ f}+\overline{\alpha'\circ f}=k\mathrm{Tra}(\alpha)+\mathrm{Tra}(\alpha')$ for $\alpha, \alpha'\in \mathrm{Hom}(H, \mathbb{F})$ and $k\in \mathbb{F}$, we have  $\mathrm{Tra}$ is a linear map.

Define
\begin{align*}
                 \mathrm{Inf}:\mathrm{H}^{2}(L/H, \mathbb{F}) & \longrightarrow \mathrm{H}^{2}(L, \mathbb{F}) \\
               \bar{\beta}& \longmapsto \bar{\beta'}.
               \end{align*}
Let $x, y\in L, \beta\in \mathrm{Z}^{2}(L/H, \mathbb{F})$. Define $\beta'(x,y)=\beta(\pi(x), \pi(y))$. Then
\begin{align*}
   &(-1)^{|x||z|}\beta'([x, y], z)+(-1)^{|y||x|}\beta'([y, z], x)+(-1)^{|z||y|}\beta'([z, x], y) \\
=&(-1)^{|x||z|}\beta(\pi([x, y]), \pi(z))+(-1)^{|y||x|}\beta(\pi([y, z]), \pi(x))+(-1)^{|z||y|}\beta(\pi([z, x]), \pi(y))\\
=&(-1)^{|x||z|}\beta([\pi(x), \pi(y)], \pi(z))+(-1)^{|y||x|}\beta([\pi(y), \pi(z)], \pi(x))+(-1)^{|z||y|}
\beta([\pi(z), \pi(x))], \pi(y)).
\end{align*}
Since $|x|=|\pi(x)|, |y|=|\pi(y)|, |z|=|\pi(z)|$ and $\beta\in \mathrm{Z}^{2}(L/H, \mathbb{F})$, we have $\beta'\in \mathrm{Z}^{2}(L, \mathbb{F})$.

Define a map $I: \mathrm{Z}^{2}(L/H, \mathbb{F})\longrightarrow\mathrm{ Z}^{2}(L, \mathbb{F})$ by $I(\beta)=\beta'$.
Suppose that $\beta\in \mathrm{B}^{2}(L/H, \mathbb{F})$. Then there exists a linear function $f:L/H\longrightarrow \mathbb{F}$ such that $\beta(\bar{x}, \bar{y})=-f([\bar{x}, \bar{y}])$ where $\bar{x}=\pi(x), \bar{y}=\pi(y)$.
 Then $$I(\beta)(\bar{x}, \bar{y})=\beta'(x, y)=\beta(\pi(x), \pi(y))=-f([\pi(x), \pi(y)])=-f\circ\pi([x,y])\in \mathrm{B}^{2}(L, \mathbb{F}).$$ We have
$I(\mathrm{B}^{2}(L/H, \mathbb{F})\subseteq \mathrm{B}^{2}(L, \mathbb{F})$. Thus, the map $I: \mathrm{Z}^{2}(L/H, \mathbb{F})\longrightarrow \mathrm{Z}^{2}(L, \mathbb{F})$ can induce a map $\mathrm{H}^{2}(L/H, \mathbb{F})\longrightarrow \mathrm{H}^{2}(L, \mathbb{F})$ defined by $\mathrm{Inf}(\bar{\beta})=\bar{\beta'}$.
Since
\begin{align*}
  I(b\beta+\tilde{\beta})(\bar{x}, \bar{y})= & (b\beta+\tilde{\beta})'(x, y) \\
  = &(b\beta+\tilde{\beta})(\pi(x), \pi(y))\\
= &b\beta(\pi(x), \pi(y))+
\tilde{\beta}(\pi(x), \pi(y))\\
= &b\beta'(x, y)+\tilde{\beta}'(x, y)\\
= &(bI(\beta)(\bar{x}, \bar{y})+I(\tilde{\beta}))(\bar{x}, \bar{y}),
\end{align*}
where $\bar{x}=\pi(x), \bar{y}=\pi(y), \beta, \tilde{\beta}\in \mathrm{Z}^{2}(L/H, \mathbb{F})$, the map $I$ is a linear map and so is $\mathrm{Inf}$.

\begin{theorem}\label{the311}(5-sequence of cohomology)
\ Let $L$ be a Lie superalgebra and $H$ a central subalgebra of $L$.
 Then
$$0\longrightarrow \mathrm{Hom}(L/H, \mathbb{F})\stackrel{\mathrm{Inf}}{\longrightarrow}\mathrm{Hom}(L, \mathbb{F})\stackrel{\mathrm{Res}}{\longrightarrow}\mathrm{Hom}(H, \mathbb{F})\stackrel{\mathrm{Tra}}{\longrightarrow}H^{2}(L/H, \mathbb{F})\stackrel{\mathrm{Inf}}{\longrightarrow}H^{2}(L, \mathbb{F})$$
is exact.
\end{theorem}

\begin{proof}
First we show that the sequence is exact at $\mathrm{Hom}(L/H, \mathbb{F})$. Let $f\in \mathrm{Hom}(L/H, \mathbb{F})$. Since
$\mathrm{Inf}(f)(x)=f\circ\pi(x)=0$ for all $x\in L$, we have $f=0$. Thus, $\mathrm{Inf}(f)$ is an injective map.

Next we show that the sequence is exact at $\mathrm{Hom}(L, \mathbb{F})$. Let $f\in \mathrm{Hom}(L/H, \mathbb{F})$. Then $\mathrm{Inf}(f)=f\circ\pi$. Since
$\mathrm{Res}(\mathrm{Inf}(f))=f\circ\pi\circ i=0$, we have $\mathrm{Im}(\mathrm{Inf})\subseteq \ker(\mathrm{Res})$.
Conversely,
let $g\in \mathrm{Hom}(L, \mathbb{F})$ such that $\mathrm{Res}(g)=g\circ i=0$. Then $g(H)=0$ and
there exists $\bar{g}\in \mathrm{Hom}(L/H, \mathbb{F})$ such that $\mathrm{Inf}(\bar{g})=\bar{g}\circ\pi=g$. Thus,  $\ker(\mathrm{Res})\subseteq\mathrm{Im}(\mathrm{Inf})$.

Next we show that the sequence is exact at $\mathrm{Hom}(H, \mathbb{F})$. Let $g\in \mathrm{Hom}(L, \mathbb{F})$, $\mu$ a section of $\pi$ and $f$ a factor set defined by $\mu$.
Then \begin{align*}
     g\circ f(\bar{x}, \bar{y})= & g([\mu(\bar{x}), \mu(\bar{y})]-\mu([\bar{x}, \bar{y}])) \\
       = & (-1)^{|g||\mu(\bar{x})|}\mu(\bar{x})\cdot g(\mu(\bar{y}))-g(\mu([\bar{x}, \bar{y}]))\\
= & -g(\mu([\bar{x}, \bar{y}]).
     \end{align*}
It follows that $g\circ f\in \mathrm{B}^{2}(L/H, \mathbb{F})$.
Since $\mathrm{Tra}(\mathrm{Res}(g))=\mathrm{Tra}(g\circ i)=\overline{g\circ i\circ f}=0$, we have $\mathrm{Im}(\mathrm{Res})\subseteq \ker(\mathrm{Tra})$.
Conversely, let $\alpha\in \mathrm{Hom}(H, \mathbb{F})$ such that $\mathrm{Tra}(\alpha)=\overline{\alpha\circ f}=0$.
Then $\alpha\circ f\in \mathrm{B}^{2}(L/H, \mathbb{F})$. Clearly, there exists a linear function $\sigma:L/H\longrightarrow\mathbb{F}$ such that
$\alpha\circ f(\bar{x}, \bar{y})=-\sigma([\bar{x}, \bar{y}])$, where $\bar{x}=\pi(x), \bar{y}=\pi(y), x, y\in L$. Since $\bar{x}=\pi(x)=\pi\circ\mu(\bar{x})$, we have $\pi(x-\mu(\bar{x}))=0$ and $x-\mu(\bar{x})\in \ker\pi=\mathrm{Im}i=H$. There exists $h_{x}\in H$ such that
$x-\mu(\bar{x})=h_{x}$ for all $x\in L$. Let $x=\mu(\bar{x})+h_{x}$ and $y=\mu(\bar{y})+h_{y}$. Then $[x, y]=[\mu(\bar{x}),\mu(\bar{y})]$. Since $[x, y]=\mu([x,y])+h_{[x, y]}$, we have $[\mu(\bar{x}), \mu(\bar{y})]-\mu([\bar{x}, \bar{y}]=h_{[x, y]}$. Then
\begin{align*}
  \alpha(h_{[x, y]})
= & \alpha([\mu(\bar{x}), \mu(\bar{y})]-\mu([\bar{x}, \bar{y}])\\
  = & \alpha\circ f(\bar{x}, \bar{y})\\
= & -\sigma([\bar{x}, \bar{y}]).
\end{align*}
Define a function $\gamma:L\longrightarrow \mathbb{F}$ by $\gamma(x)=\alpha(h_{x})+\sigma(\bar{x})$. Since $\alpha, \sigma$ are linear functions, we have $\gamma$ is a linear function. Since $\gamma([x, y])=\alpha(h_{[x, y]})+\sigma([\bar{x}, \bar{y}])=0$,
 we have $\gamma([x, y])=(-1)^{|\gamma||x|}x\cdot\gamma(y)$. Thus, $\gamma\in \mathrm{Hom}(L, \mathbb{F})$. Note that $\gamma(h)=\alpha(h)+\sigma(\bar{h})=\alpha(h)$ for all $h\in H$. We have $\mathrm{Res}(\gamma)=\alpha$ and $\ker(\mathrm{Tra})\subseteq \mathrm{Im}(\mathrm{Res})$. This implies the sequence is exact at $\mathrm{Hom}(H, \mathbb{F})$.

Finally, we show that the sequence is exact at $\mathrm{H}^{2}(L/H, \mathbb{F})$. Suppose $\alpha\in \mathrm{Hom}(H, \mathbb{F})$ such that $\mathrm{Tra}(\alpha)=\overline{\alpha\circ f}$, where $f$ is a factor set defined by $\mu$
and $\alpha\circ f\in \mathrm{Z}^{2}(L/H, \mathbb{F})$. Then $\mathrm{Inf}(\overline{\alpha\circ f})=\overline{(\alpha\circ f)'}$, where $(\alpha\circ f)'(x, y)=\alpha\circ f(\pi(x), \pi(y))$ for $x, y\in L$. We will show $(\alpha\circ f)'\in \mathrm{B}^{2}(L, \mathbb{F})$.

Let $x=\mu(\bar{x})+h_{x}$ and $y=\mu(\bar{y})+h_{y}$. Then $[x, y]=[\mu(\bar{x}),\mu(\bar{y})]$. Since $[x, y]=\mu([x,y])+h_{[x, y]}$, we have $[\mu(\bar{x}), \mu(\bar{y})]-\mu([\bar{x}, \bar{y}]=h_{[x, y]}$. Hence $\alpha\circ f(\bar{x}, \bar{y})
=\alpha([\mu(\bar{x}), \mu(\bar{y})]-\mu([\bar{x}, \bar{y}])=\alpha(h_{[x, y]})$. Define a function $\theta: L\longrightarrow \mathbb{F}$ by $\theta(x)=-\alpha(h_{x})$. Clearly, $\theta$ is a linear function. Note that $\theta([x, y])=-\alpha(h_{[x, y]})=-
\alpha\circ f(\bar{x}),\mu(\bar{y})=-(\alpha\circ f)'(x, y)$. We have $(\alpha\circ f)'\in \mathrm{B}^{2}(L, \mathbb{F})$. Thus, $\overline{(\alpha\circ f)'}
=0$ and $\mathrm{Im}(\mathrm{Tra})\subseteq\ker(\mathrm{Inf})$.
Conversely, let $\bar{g}\in \ker(\mathrm{Inf})$
for some $g\in \mathrm{Z}^{2}(L/H, \mathbb{F})$. Then $g(\bar{x}, \bar{y})=g'(x, y)=-\theta([x, y])$ for some linear function $\theta: L\longrightarrow \mathbb{F}$. Since $\theta$ is a linear function, we have $\theta\circ f\in \mathrm{Z}^{2}(L/H, \mathbb{F})$ for some factor set. Let $x=\mu(\bar{x})+h_{x}$. Then $[x, y]=[\mu(\bar{x}), \mu(\bar{y})]$. We have
\begin{align*}
g'(x, y)= & g(\bar{x}, \bar{y})\\= & -\theta([x, y])\\
= & -\theta
([\mu(\bar{x}), \mu(\bar{y})])\\
= & -\theta([\mu(\bar{x}), \mu(\bar{y})])+\theta(\mu([\bar{x}, \bar{y}]))-\theta(\mu([\bar{x}, \bar{y}]))\\
= & -\theta\circ f(\bar{x}, \bar{y})-\theta\circ\mu([\bar{x}, \bar{y}]),
\end{align*}
where $\theta\circ\mu: L/H\longrightarrow \mathbb{F}$, Since $g$ and $-\theta\circ f $ are equivalent and $\bar{g}=\overline{-\theta\circ f}=-\mathrm{Tra}(\theta)$, we have $\ker(\mathrm{Inf})\subseteq\mathrm{Im}(\mathrm{Tra})$ and this implies the sequence is exact at $\mathrm{H}^{2}(L/H, \mathbb{F})$.
\end{proof}

Now we will use 5-sequences of cohomology to prove that the multiplier of a Lie superalgebra is isomorphic to the second cohomology group.

\begin{lemma}\label{lem312}
Let $Z$ be  a  central ideal of $L$. Then $L^{2}\cap Z$ is isomorphic to the image of $\mathrm{Hom}(Z, \mathbb{F})$ under the map
 $\mathrm{Tra}$. In particular, if the map $\mathrm{Tra}$ is surjective, then $L^{2}\cap Z\cong \mathrm{H}^{2}(L/Z, \mathbb{F}).$
\end{lemma}
\begin{proof}
Let $0\longrightarrow Z\longrightarrow L\longrightarrow L/Z\longrightarrow0$ be the natural exact sequence. By Theorem \ref{the311}, we have the sequence
$$\mathrm{Hom}(L, \mathbb{F})\stackrel{\mathrm{Res}}{\longrightarrow}\mathrm{Hom}(Z, \mathbb{F})\stackrel{\mathrm{Tra}}{\longrightarrow}\mathrm{H}^{2}(L/Z, \mathbb{F})$$
is exact. Let $J=\mathrm{Im}(\mathrm{Res})=\ker(\mathrm{Tra})$. Then $\mathrm{Hom}(Z, \mathbb{F})/J\cong \mathrm{Im}(\mathrm{Tra})$. If the map $\mathrm{Tra}$ is surjective, then $L^{2}\cap Z\cong \mathrm{H}^{2}(L/Z, \mathbb{F})$.
We only show that $\mathrm{Hom}(Z, \mathbb{F})/J\cong L^{2}\cap Z$. Note that $L^{2}\cap Z$ is abelian and it is isomorphic to the dual superspace $\mathrm{Hom}(L^{2}\cap Z, \mathbb{F})$, consider the restriction homomorphism $\mathrm{res}:\mathrm{Hom}(Z, \mathbb{F})\longrightarrow \mathrm{Hom}(L^{2}\cap Z, \mathbb{F})$. Since $Z$ and $L^{2}\cap Z$ are abelian, we have $\mathrm{res}$ is surjective and $\mathrm{Hom}(Z, \mathbb{F})/\ker(\mathrm{res}) \cong \mathrm{Hom}(L^{2}\cap Z, \mathbb{F})$. Thus we only need to show $J\cong\ker(\mathrm{res})$.

Suppose $f\in J$ and $\hat{f}\in \mathrm{Hom}(L, \mathbb{F})$ such that $\mathrm{Res}(\hat{f})=f$. Since $\mathbb{F}$ is a trivial $L$-module,
we have $\hat{f}([x, y])=(-1)^{|x||f|}x\cdot\hat{f}(y)=0$ for all $x, y\in L$. It follows that $L^{2}\subseteq \ker \hat{f}$. Hence $L^{2}\cap Z\subseteq \ker f$ and $f\in \ker(\mathrm{res})$. Thus, $J\subseteq\ker(\mathrm{res})$. Conversely, let $f\in \ker(\mathrm{res})$. Then $f\in \mathrm{Hom}(Z, \mathbb{F})$ and $f(L^{2}\cap Z)=0$. Hence $f$ induces a homomorphism $f_{1}: Z/(L^{2}\cap Z)\longrightarrow \mathbb{F}$, where $f_{1}(z+L^{2}\cap Z)=f(z)$ for $z\in Z$. Since $Z/(L^{2}\cap Z)\cong (Z+L^{2})/L^{2}$, we can get a homomorphism $f_{2}:(Z+L^{2})/L^{2}\longrightarrow \mathbb{F}$, where $f_{2}(z+L^{2})=f_{1}(z+L^{2}\cap Z)$ for $z\in Z$. Consider the homomorphism $(Z+L^{2})/L^{2}\longrightarrow L/L^{2}$, there exists a homomorphism $f_{3}: L/L^{2}\longrightarrow \mathbb{F}$, where $f_{3}(x+L^{2})=f_{2}(x+L^{2})$ for $x\in Z$. Since $L/L^{2}$ is abelian, $f_{3}$ can be extended to $\hat{f}: L\longrightarrow \mathbb{F}$,
where $\hat{f}(x)=f_{3}(x+L^{2})$. Thus, $\hat{f}\in J$ and $\ker(\mathrm{res})\subseteq J$.
\end{proof}

\begin{lemma}\label{lem313}
Let $0\longrightarrow R\longrightarrow F\stackrel{\pi}{\longrightarrow}L\longrightarrow 0$ be a free presentation of a Lie superalgebra $L$.  View $L$ as the factor algebra $(F/[F, R])/(R/[F, R])$. Then the map $\mathrm{Tra}:\mathrm{Hom}(R/[F, R], \mathbb{F})\longrightarrow \mathrm{H}^{2}(L, \mathbb{F})$ is surjective.
\end{lemma}

\begin{proof}
Let $\bar{\alpha}\in \mathrm{H}^{2}(L, \mathbb{F})$ and $0\longrightarrow\mathbb{F}\longrightarrow L'\stackrel{\theta}{\longrightarrow}L\longrightarrow 0$ be a central extension associated with $\bar{\alpha}$. By \cite[Theorem 2.5]{20}, there exists a Lie superalgebra homomorphism
$\beta: F/[F, R]\longrightarrow L'$ such that the following diagram is commutative:
$$\xymatrix{0\ar[r]&\dfrac{R}{[F, R]}\ar[d]_{\beta_{1}}\ar[r]&\dfrac{F}{[F, R]}\ar[d]_{\beta}\ar[r]^{\tilde{\pi}}&L\ar@{=}[d]_{\mathrm{id}}\ar[r]&0\\
0\ar[r]&\mathbb{F}\ar[r]&L'\ar[r]^{\theta}&L\ar[r]&0,}$$
where $\beta_{1}$ is the restriction of $\beta$ to $R/[F, R]$. Then $\beta_{1}\in \mathrm{Hom}(R/[F, R], \mathbb{F})$. We claim that $\mathrm{Tra}(\beta_{1})=\bar{\alpha}$ and the map $\mathrm{Tra}$ is surjective. In fact, let $\mu$ be a section of $\tilde{\pi}$. Then $\theta\circ\beta\circ\mu=\tilde{\pi}\circ\mu=\mathrm{id}_{L}$ and $\beta\circ\mu$ is a section of $\theta$. Let $\lambda=\beta\circ\mu$ and $g(x, y)=[\lambda(x), \lambda(y)]-\lambda([x, y])$ for $x, y\in L$. Then $g\in \mathrm{Z}^{2}(L, \mathbb{F})$ and $\bar{g}=\bar{\alpha}$ in $\mathrm{H}^{2}(L, \mathbb{F})$.
Since
\begin{align*}
  g(x, y)= &[\lambda(x), \lambda(y)]-\lambda([x, y])  \\
  = &[\beta\circ\mu(x), \beta\circ\mu(y)]-\beta\circ\mu([x, y])\\
  = &\beta([\mu(x), \mu(y)]-\mu([x, y]))\\
  = &\beta_{1}([\mu(x), \mu(y)]-\mu([x, y]))\\
  = &\beta_{1}(f(x, y)),
\end{align*}
where $f(x, y)=[\mu(x), \mu(y)]-\mu([x, y])$, we have $\mathrm{Tra}(\beta_{1})=\overline{\beta_{1}\circ f}=\bar{g}=\bar{\alpha}$ as we claimed.
\end{proof}

\begin{theorem}\label{the314}
Let $0\longrightarrow R\longrightarrow F\stackrel{\pi}{\longrightarrow}L\longrightarrow 0$ be a free presentation of a Lie superalgebra $L$. Then $\mathrm{H}^{2}(L, \mathbb{F})\cong (F^{2}\cap R)/[F, R]$. In particular, $\mathcal{M}(L)\cong \mathrm{H}^{2}(L, \mathbb{F})$.
\end{theorem}
\begin{proof}
Note that $0\longrightarrow \frac{R}{[F, R]}\longrightarrow \frac{F}{[F, R]}\stackrel{\tilde{\pi}}{\longrightarrow} L\longrightarrow 0$ be a central extension of a Lie superalgebra $L$, where $\tilde{\pi}$ is induced by $\pi$. By Lemma \ref{lem313}, the map $\mathrm{Tra}:\mathrm{Hom}(R/[F, R], \mathbb{F})\longrightarrow \mathrm{H}^{2}(L, \mathbb{F})$ is surjective. Then by Lemma \ref{lem312}, we have $$(F/[F, R])^{2}\cap R/[F, R] \cong \mathrm{H}^{2}((F/[F, R])/(R/[F, R]), \mathbb{F})\cong \mathrm{H}^{2}(L, \mathbb{F}).$$ Since $$(F/[F, R])^{2}\cap R/[F, R]\cong F^{2}/[F, R]\cap R/[F, R]\cong (F^{2}\cap R)/[F, R],$$ we have $\mathrm{H}^{2}(L, \mathbb{F})\cong (F^{2}\cap R)/[F, R]$.  By Definition \ref{def214}, we have $\mathcal{M}(L)=(F^{2}\cap R)/[F, R]\cong \mathrm{H}^{2}(L, \mathbb{F})$.
\end{proof}

%%%%%%%%%%%%%%%%%%%%%%%%%%%%%%%%%%%%%%%%%%%%%%%%%%%%%%%%%%%%%%%%%%%%%%%%%%%%%%%%%%%%


\begin{thebibliography}{99}
\bibitem{1} I. Schur. Uber die darstellung der endlichen gruppen durch gebrochene lineare substitutionen. J. fur Math. 1904, 127: 20-50.
\bibitem{2} G. Karpilovsky. The schur multiplier. London Math. Soc. 1987.
\bibitem{3} P. Batten. Covers and multipliers of Lie algebras. Dissertation. North Carolina State University. 1993, 1-41.
\bibitem{4} P. Batten. Ernest Stitzinger. On covers of Lie algebras. Comm. Algebra 1996, 24(14): 4301-4317.
\bibitem{5} K. Moneyhun. Isoclinisms in Lie algebras. Algebras Groups Geom. 11. 1994, 9-22.
\bibitem{6} A. R. Salemkar, V. Alamian and H. Mohammadzadeh. Some properties of the schur multiplier and covers of Lie algebras. Comm. Algebra 2008, 36: 697-707.
\bibitem{7} A. R. Salemkar, B. Edalatzadeh and H. Mohammadzadeh. On covers of perfect Lie algebras. Algebra Colloq. 2011, 18(3): 419-427.
\bibitem{8} P. Batten, K. Moneyhun and E. Stitzinger. On characterizing nilpotent Lie algebras by their multipliers.  Comm. Algebra 1996, 24(14): 4319-4330.
\bibitem{9} P. Hardy, E. Stitzinger. On characterizing nilpotent Lie algebras by their multipliers, t(L)=3, 4, 5, 6. Comm. Algebra. 1998, 26(11): 3527-3539.
\bibitem{10} P. Hardy. On characterizing nilpotent Lie algebras by their multipliers, III. Comm. Algebra 2005, 33: 4205-4210.
\bibitem{11} P. Niroomand. On dimension of the schur multiplier of nilpotent Lie algebras. Cent. Eur. J. Math. 2011, 9(1): 57-64.
\bibitem{12} P. Niroomand, F. Russo. A note on the schur multiplier of a nilpotent Lie algebra. Comm. Algebra 2011, 39: 1293-1297.
\bibitem{13} A. R. Salemkar, S. A. Niri. Bounds for the dimension  of the  schur multiplier of a pair of nilpotent  Lie algebras. Asian-Eur. J. Math. 2012, 5(4) 1250059 (9 pages).
\bibitem{14} M. Araskhan. The dimension of the c-nilpotent multiplier. J. Algebra 2013, 386: 105-112.
\bibitem{15} Z. Riyahi, A. R. Salemkar. A remark on the schur multiplier of nilpotent Lie algebras. J. Algebra 2015, 438: 1-6.
\bibitem{16} F. Saeedi, H. Arabyani and P. Niroomand. On dimension of schur multiplier of nilpotent Lie algebras II. Asian-Eur. J. Math. 2017, 10(4) 1750076 (8 pages).
\bibitem{17} S. Nayak. Multipliers of nilpotent Lie superalgebras. arXiv:1801.03798, 2018.
\bibitem{18} S. Nayak. Isoclinism in Lie superalgebras. arXiv:1804.10434, 2018.
\bibitem{19} Y. L. Zhang, W. D. Liu. Classification of nilpotent Lie superalgebras of multiplier-rank $\leq$ 2. arXiv:1807.08884, 2018.
\bibitem{20} X. X. Miao, W. D. Liu. Multipliers, covers and stem extensions for Lie superalgebras. arXiv:1808.02991v2 [math.RA], 2018.
\bibitem{21} R. Lal. Algebra 2. Infosys Science Foundation Series Mathematical. Springer. 2017: 389-417.
\bibitem{22} I. M. Musson, Lie Superalgebras and enveloping Algebras, GSM 131, Amer. Math. Soc.
Providence, RI, 2012: 355-380.
\bibitem{23} C. Chevalley, S. Eilenberg. Cohomlogy theory of Lie groups and Lie algebras, Trans. Amer. Math. Soc. 1948, 63: 120-122.
\end{thebibliography}
\end{document}